\documentclass[11pt,draft]{amsart}

\usepackage{cite,amsmath,amssymb,amsthm,latexsym,a4}

\usepackage[cp1251]{inputenc}

\newtheorem{theorem}{Theorem}[section]

\newtheorem{proposition}[theorem]{Proposition}

\newtheorem{lemma}[theorem]{Lemma}

\newtheorem{remark}[theorem]{Remark}

\DeclareMathOperator{\con}{con}

\DeclareMathOperator{\gr}{gr}

\DeclareMathOperator{\Perm}{Perm}

\DeclareMathOperator{\Stab}{Stab}

\DeclareMathOperator{\var}{var}

\numberwithin{equation}{section}

\makeatletter

\renewcommand*\subjclass[2][2010]{\def\@subjclass{#2}\@ifundefined{subjclassname@#1}{\ClassWarning{\@classname}{Unknown edition (#1) of Mathematics Subject Classification; using '2010'.}}{\@xp\let\@xp\subjclassname\csname subjclassname@#1\endcsname}}

\makeatother

\begin{document}

\title[On modular elements of the lattice of semigroup varieties]{On modular and cancellable elements\\
of the lattice of semigroup varieties}

\author{D.~V.~Skokov and B.~M.~Vernikov}

\address{Ural Federal University, Institute of Natural Sciences and Mathematics, Lenina 51, 620083 Ekaterinburg, Russia}

\email{bvernikov@gmail.com, dmitry.skokov@gmail.com}

\date{}

\thanks{The work was partially supported by the Ministry of Education and Science of the Russian Federation (project 1.6018.2017/8.9) and by the Russian Foundation for Basic Research (grant No.\,17-01-00551).}

\begin{abstract}
We completely determine all semigroup varieties satysfiyng a permutational identity of length 3 that are modular elements of the lattice of all semigroup varieties. Using this result, we provide an example of a semigroup variety that is a modular but not cancellable element of this lattice.
\end{abstract}



\maketitle

\section{Introduction and summary}
\label{intr}

There are a number of articles devoted to an examination of special elements in the lattice \textbf{SEM} of all semigroup varieties (see surveys~\cite[Section~14]{Shevrin-Vernikov-Volkov-09} and~\cite{Vernikov-15}). Here we continue these considerations and examine special elements of two types, namely, modular and cancellable elements. Recall that an element $x$ of a lattice $\langle L;\vee,\wedge\rangle$ is called
\begin{align*}
&\text{\emph{modular} if}\quad&&(\forall y,z\in L)\quad \bigl(y\le z\longrightarrow(x\vee y)\wedge z=(x\wedge z)\vee y\bigr),\\
&\text{\emph{cancellable} if}\quad&&(\forall y,z\in L)\quad (x\vee y=x\vee z\ \&\ x\wedge y=x\wedge z\longrightarrow y=z).
\end{align*}
It is evident that every cancellable element of a lattice is modular. A valuable information about modular and cancellable elements in abstract lattices can be found in~\cite{Seselja-Tepavcevic-01}, for instance.

Several results about modular elements of the lattice \textbf{SEM} were provided in the papers~\cite{Jezek-McKenzie-93,Vernikov-07,Shaprynskii-12}. In particular, commutative semigroup varieties that are modular elements in \textbf{SEM} are completely determined in~\cite[Theorem~3.1]{Vernikov-07}. Further, it is verified in~\cite{Gusev-Skokov-Vernikov-17+} that the properties of being modular and cancellable elements of \textbf{SEM} are equivalent in the class of commutative semigroup varieties. The objective of this article is to prove that this equivalence is not the case in slightly wider class, namely in the class of semigroup varieties satisfying a permutational identity of length~3.

A semigroup variety is called a \emph{nil-variety} if it consists of nilsemigroups. Semigroup words unlike letters are written in bold. Two sides of identities we connect by the symbol~$\approx$, while the symbol~$=$ stands for the equality relation on the free semigroup. As usual, we write the pair of identities $x\mathbf{u\approx u}x\approx\mathbf u$ where the letter $x$ does not occur in the word \textbf u in the short form $\mathbf u\approx 0$ and refer to the expression $\mathbf u\approx 0$ as to a single identity. A \emph{permutational} identity is an identity of the form
\begin{equation}
\label{permut id}
x_1x_2\cdots x_n\approx x_{1\pi}x_{2\pi}\cdots x_{n\pi}
\end{equation}
where $\pi$ is a non-trivial permutation on the set $\{1,2,\dots,n\}$. The number $n$ is called a \emph{length} of the identity~\eqref{permut id}. We denote by \textbf T the trivial semigroup variety and by \textbf{SL} the variety of all semilattices.

The main result of this note is the following

\begin{theorem}
\label{permut-3}
A semigroup variety $\mathbf V$ satisfying a permutational identity of length~$3$ is a modular element in the lattice $\mathbf{SEM}$ if and only if $\mathbf V=\mathbf{M\vee N}$ where $\mathbf M$ is one of the  varieties $\mathbf T$ or $\mathbf{SL}$, while the variety $\mathbf N$ satisfies one of the following identity systems:
\begin{align}
\label{xyz=zyx,xxy=0}&xyz\approx zyx,\,x^2y\approx 0;\\
\label{xyz=yzx,xxy=0}&xyz\approx yzx,\,x^2y\approx 0;\\
\label{xyz=yxz,xyy=0}&xyz\approx yxz,\,xyzt\approx xzty,\,xy^2\approx 0;\\
\label{xyz=xzy,xxy=0}&xyz\approx xzy,\,xyzt\approx yzxt,\,x^2y\approx 0.
\end{align}
\end{theorem}

We denote by $\mathbb S_n$ the full permutation group on the set $\{1,2,\dots,n\}$. The semigroup variety given by an identity system $\Sigma$ is denoted by $\var\Sigma$. In~\cite[Question~3.2]{Gusev-Skokov-Vernikov-17+}, the question is asked whether there is a semigroup variety that is a modular but not cancellable element of the lattice \textbf{SEM}. The following result gives the affirmative answer to this question.

\begin{proposition}
\label{modular not cancellable}
Let $\rho$ be a non-trivial permutation from $\mathbb S_3$. The variety $\var\{xyzt\approx xyx\approx x^2\approx 0,\,x_1x_2x_3\approx x_{1\rho}x_{2\rho}x_{3\rho}\}$ is a modular but not cancellable element of the lattice $\mathbf{SEM}$.
\end{proposition}

The article consists of three sections. Section~\ref{prel} contains auxiliary results. In Section~\ref{proof} we verify Theorem~\ref{permut-3} and Proposition~\ref{modular not cancellable}.

\section{Preliminaries}
\label{prel}

We denote by $F$ the free semigroup over a countably infinite alphabet. If $\mathbf u\in F$ then $\con(\mathbf u)$ denotes the set of all letters occurring in the word \textbf u. The following claim is well known and easily verified.

\begin{lemma}
\label{u=v in SL}
An identity $\mathbf{u\approx v}$ holds in the variety $\mathbf{SL}$ if and only if $\con(\mathbf u)=\con(\mathbf v)$.\qed
\end{lemma}

An element $x$ of a lattice $\langle L;\vee,\wedge\rangle$ is called \emph{neutral} if, for any $y,z\in L$, the sublattice of $L$ generated by $x$, $y$ and $z$ is distributive. It is well known that the variety \textbf{SL} is a neutral element of the lattice \textbf{SEM}~\cite[Proposition~4.1]{Volkov-05} and an atom of this lattice (see the survey~\cite{Shevrin-Vernikov-Volkov-09}, for instance). This fact together with~\cite[Corollary~2.1]{Shaprynskii-11} immediately imply the following

\begin{lemma}
\label{+-SL}
For a semigroup variety $\mathbf V$, the following are equivalent:
\begin{itemize}
\item[\textup{(i)}] the variety $\mathbf V$ is a modular element of the lattice $\mathbf{SEM}$;
\item[\textup{(ii)}] the variety $\mathbf{SL\vee V}$ is a modular element of the lattice $\mathbf{SEM}$;
\item[\textup{(iii)}] the variety $\mathbf{SL\vee V}$ is a modular element of the coideal $[\mathbf{SL})$ of the lattice $\mathbf{SEM}$.\qed
\end{itemize}
\end{lemma}

The following claim gives a strong necessary condition for a semigroup variety to be a modular element in the lattice \textbf{SEM}.

\begin{proposition}
\label{nec}
If $\mathbf V$ is a modular  element of the lattice $\mathbf{SEM}$ then either $\mathbf V$ coincides with the variety of all semigroups or $\mathbf V=\mathbf{M\vee N}$ where $\mathbf M$ is one of the varieties $\mathbf T$ or $\mathbf{SL}$, while  $\mathbf N$ is a nil-variety.\qed
\end{proposition}

This proposition was proved (in slightly weaker form and in some other terminology) in~\cite[Proposition~1.6]{Jezek-McKenzie-93}. It was formulated in the form given here in~\cite[Proposition~2.1]{Vernikov-07}. A direct and transparent proof of Proposition~\ref{nec} not depending on a technique from~\cite{Jezek-McKenzie-93} is provided in~\cite{Shaprynskii-12}.

As we have already mentioned, commutative varieties that are modular elements in \textbf{SEM} are completely classified in~\cite[Theorem~3.1]{Vernikov-07}. We need the following consequence of this result that was verified in~\cite{Vernikov-Volkov-06}, in fact.

\begin{proposition}
\label{commut}
If a semigroup variety satisfies the identities $x^2y\approx 0$ and $xy\approx yx$ then it is a modular element in $\mathbf{SEM}$.\qed
\end{proposition}

We need also the following claim that is a part of~\cite[Theorem~4.5]{Vernikov-07}.

\begin{lemma}
\label{permut-3 nec}
Let $\mathbf V$ be a nil-variety of semigroups satisfying an identity of the form $x_1x_2x_3\approx x_{1\pi}x_{2\pi}x_{3\pi}$ where $\pi$ is a non-trivial permutation from $\mathbb S_3$. If $\mathbf V$ is a modular element of the lattice $\mathbf{SEM}$ then $\mathbf V$ satisfies also:
\begin{itemize}
\item[\textup{(i)}] all permutational identities of length~$4$;
\item[\textup{(ii)}] the identity $xy^2\approx 0$ whenever $\pi=(12)$;
\item[\textup{(iii)}] the identity $x^2y\approx 0$ whenever $\pi$ is one of the permutations $(13)$, $(23)$ or $(123)$.\qed
\end{itemize}
\end{lemma}

For a natural number $n$ and a semigroup variety \textbf V, we denote by $\Perm_n(\mathbf V)$ the set of all permutations $\pi\in\mathbb S_n$ such that \textbf V satisfies the identity~\eqref{permut id}. Clearly, $\Perm_n(\mathbf V)$ is a subgroup in $\mathbb S_n$. If $1\le i\le n$ then we denote by $\Stab_i(n)$ the set of all permutations $\pi\in\mathbb S_n$ with $i\pi=i$. Obviously, $\Stab_i(n)$ also is a subgroup in $\mathbb S_n$. Moreover, it is well known that $\Stab_i(n)$ is a maximal proper subgroup in $\mathbb S_n$. We need the following partial cases of results of the article~\cite{Pollak-73}.

\begin{lemma}
\label{permut conseq}
Let $\mathbf V$ be a semigroup variety.
\begin{itemize}
\item[\textup{1)}] If $\mathbf V$ satisfies a non-trivial identity of the form $x_1x_2x_3\approx x_{1\pi}x_{2\pi}x_{3\pi}$ and $n\ge 4$ then
\begin{itemize}
\item[\textup{(i)}] $\Perm_n(\mathbf V)\supseteq\Stab_n(n)$ whenever $\pi=(12)$;
\item[\textup{(ii)}] $\Perm_n(\mathbf V)\supseteq\Stab_n(1)$ whenever $\pi=(23)$;
\item[\textup{(iii)}] $\Perm_n(\mathbf V)=\mathbb S_n$ otherwise.
\end{itemize}
\item[\textup{2)}] If $\mathbf V$ satisfies the identity $xyzt\approx xzty$ and $n\ge 5$ then $\Perm_n(\mathbf V)\supseteq\Stab_n(1)$.\qed
\end{itemize}
\end{lemma}

If $\mathbf a\in F$ then we denote by $\ell(\mathbf a)$ the length of the word \textbf a. The following lemma readily follows from the proof of~\cite[Lemma~1]{Sapir-Sukhanov-81}.

\begin{lemma}
\label{x_1x_2...x_n=w}
If a nil-variety of semigroups satisfies an identity of the form $x_1x_2\cdots x_n\approx\mathbf w$ where $\ell(\mathbf w)\ne n$ then it satisfies also the identity $x_1x_2\cdots x_n\approx 0$.\qed
\end{lemma}

If $\mathbf a,\mathbf b\in F$ and \textbf b may be obtained from \textbf a by renaming of letters then we say that the words \textbf a and \textbf b are \emph{similar}. 

\begin{lemma}
\label{Perm_n}
Let $\mathbf V_1$ and $\mathbf V_2$ be semigroup varieties and $n$ be a natural number. Then:
\begin{itemize}
\item[\textup{(i)}] $\Perm_n(\mathbf V_1\vee\mathbf V_2)=\Perm_n(\mathbf V_1)\wedge\Perm_n(\mathbf V_2)$;
\item[\textup{(ii)}] if $\mathbf V_1$ and $\mathbf V_2$ are nil-varieties then $\Perm_n(\mathbf V_1\wedge\mathbf V_2)=\Perm_n(\mathbf V_1)\vee\Perm_n(\mathbf V_2)$.
\end{itemize}
\end{lemma}

\begin{proof}
For brevity, we will denote the identity~\eqref{permut id} by $p_n[\pi]$. If \textbf a is a word and $\pi$ is a permutation on the set $\con(\mathbf a)$ then we denote by $\pi[\mathbf a]$ the word that is obtained from \textbf a by the substitution $x\mapsto\pi(x)$ for every letter $x\in \con(\mathbf a)$.

\smallskip

(i) Let $\pi\in\Perm_n(\mathbf V_1\vee\mathbf V_2)$. Then the variety $\mathbf V_1\vee\mathbf V_2$ satisfies the identity $p_n[\pi]$. Hence this identity holds in both the varieties $\mathbf V_1$ and $\mathbf V_2$. Therefore, $\pi\in\Perm_n(\mathbf V_1)$ and $\pi\in\Perm_n(\mathbf V_2)$, whence $\pi\in\Perm_n(\mathbf V_1)\wedge\Perm_n(\mathbf V_2)$. Thus, $\Perm_n(\mathbf V_1\vee\mathbf V_2)\subseteq\Perm_n(\mathbf V_1)\wedge\Perm_n(\mathbf V_2)$. Suppose now that $\pi\in\Perm_n(\mathbf V_1)\wedge\Perm_n(\mathbf V_2)$. Then the identity $p_n[\pi]$ holds in both the varieties $\mathbf V_1$ and $\mathbf V_2$. Therefore, it holds in $\mathbf V_1\vee\mathbf V_2$. Thus, $\pi\in\Perm_n(\mathbf V_1\vee\mathbf V_2)$ and $\Perm_n(\mathbf V_1)\wedge\Perm_n(\mathbf V_2)\subseteq\Perm_n(\mathbf V_1\vee\mathbf V_2)$. This implies the required equality. 
 
\smallskip

(ii) Let $\pi\in\Perm_n(\mathbf V_1)\vee\Perm_n(\mathbf V_2)$. Then there is a sequence of permutations $\pi_1,\pi_2,\dots,\pi_m$ such that $\pi=\pi_1\pi_2\cdots \pi_m$ and, for each $i=1,2,\dots,m$, the permutation $\pi_i$ lies in one of the groups $\Perm_n(\mathbf V_1)$ or $\Perm_n(\mathbf V_2)$. Put $\mathbf u_0=x_1x_2\cdots x_n$. For each $i=1,2,\dots,m$, we define by induction the word $\mathbf u_i$ by the equality $\mathbf u_i=\pi_i[\mathbf u_{i-1}]$. It is clear that $\mathbf u_m=\pi[\mathbf u_0]$ and the sequence of words $\mathbf u_0$, $\mathbf u_1$, \dots, $\mathbf u_m$ is a deduction of the identity $p_n[\pi]$ from identities of the varieties $\mathbf V_1$ and $\mathbf V_2$. This means that $\pi\in\Perm_n(\mathbf V_1\wedge\mathbf V_2)$. Thus, $\Perm_n(\mathbf V_1)\vee\Perm_n(\mathbf V_2)\subseteq\Perm_n(\mathbf V_1\wedge\mathbf V_2)$.

It remains to verify that $\Perm_n(\mathbf V_1\wedge\mathbf V_2)\subseteq\Perm_n(\mathbf V_1)\vee\Perm_n(\mathbf V_2)$. Let $\pi\in\Perm_n(\mathbf V_1\wedge\mathbf V_2)$. Then the identity $p_n[\pi]$ holds in the variety $\mathbf V_1\wedge\mathbf V_2$. Let $\mathbf u_0$, $\mathbf u_1$, \dots, $\mathbf u_m$ be a deduction of this identity from identities of the varieties $\mathbf V_1$ and $\mathbf V_2$. In other words, $\mathbf u_0=x_1x_2\cdots x_n$, $\mathbf u_m=x_{1\pi}x_{2\pi}\cdots x_{n\pi}$ and, for each $i=0,1,\dots,m-1$, the identity $\mathbf u_i\approx\mathbf u_{i+1}$ holds in either $\mathbf V_1$ or $\mathbf V_2$. We will assume that $\mathbf u_0$, $\mathbf u_1$, \dots, $\mathbf u_m$ is the shortest sequence of words with the mentioned properties. Suppose that there is an index $i$ such that either $\mathbf u_i$ is non-linear or $\con(\mathbf u_i)\ne\{x_1,x_2,\dots,x_n\}$. Let $i$ be the least index with such a property. Clearly, $0<i<m$. The word $\mathbf u_{i-1}$ is similar to $x_1x_2\cdots x_n$. If $\con(\mathbf u_i)\ne\{x_1,x_2,\dots,x_n\}$ then there is a letter $x_j$ that occurs in one of the words $\mathbf u_{i-1}$ or $\mathbf u_i$ but does not occur in another one. The identity $\mathbf u_{i-1}\approx\mathbf u_i$ holds in one of the varieties $\mathbf V_1$ or $\mathbf V_2$, say, in $\mathbf V_1$. One can substitute~0 for $x_j$ in $\mathbf u_{i-1}\approx\mathbf u_i$. We obtain that $x_1x_2\cdots x_n\approx 0$ holds in $\mathbf V_1$, whence $\mathbf V_1$ satisfies the identity $p_n[\pi]$. In view of Lemma~\ref{x_1x_2...x_n=w}, the same is true whenever $\con(\mathbf u_i)=\{x_1,x_2,\dots,x_n\}$ but $\mathbf u_i$ is non-linear. But then the sequence of words $\mathbf u_0$, $\mathbf u_1$, \dots, $\mathbf u_{i-1}$, $\mathbf u_m$ is a deduction of the identity $p_n[\pi]$ from identities of the varieties $\mathbf V_1$ and $\mathbf V_2$ shorter than the deduction $\mathbf u_0$, $\mathbf u_1$, \dots, $\mathbf u_m$. Therefore, the words $\mathbf u_0$, $\mathbf u_1$, \dots, $\mathbf u_m$ are linear and $\con(\mathbf u_i)=\{x_1,x_2,\dots,x_n\}$ for all $i=0,1,\dots,m$. Hence there are permutations $\pi_1,\pi_2,\dots,\pi_m\in\mathbb S_n$ such that $\mathbf u_i=\pi_i[\mathbf u_{i-1}]$ for all $i=1,2,\dots,m$ and each of the permutations $\pi_1,\pi_2,\dots,\pi_m$ lies in either $\Perm_n(\mathbf V_1)$ or $\Perm_n(\mathbf V_2)$. Clearly, $\pi=\pi_1\pi_2\cdots \pi_m$, whence $\pi\in\Perm_n(\mathbf V_1)\vee\Perm_n(\mathbf V_2)$. 
\end{proof}

To verify Proposition~\ref{modular not cancellable}, we need the information about the structure of the subvariety lattice of the group $\mathbb S_3$. It is generally known and easy to check that this lattice has the form shown in Fig.~\ref{Sub(S_3)}. Here $\mathbb T$ is the trivial group and $\gr\{\pi\}$ denotes the subgroup of $\mathbb S_3$ generated by the permutation $\pi$.

\begin{figure}[tbh]
\begin{center}
\unitlength=1mm
\special{em:linewidth 0.4pt}
\linethickness{0.4pt}
\begin{picture}(62,36)
\put(1,18){\circle*{1.33}}
\put(21,18){\circle*{1.33}}
\put(31,3){\circle*{1.33}}
\put(31,33){\circle*{1.33}}
\put(41,18){\circle*{1.33}}
\put(61,18){\circle*{1.33}}
\put(1,18){\line(2,1){30}}
\put(1,18){\line(2,-1){30}}
\put(21,18){\line(2,3){10}}
\put(21,18){\line(2,-3){10}}
\put(31,3){\line(2,3){10}}
\put(31,3){\line(2,1){30}}
\put(31,33){\line(2,-3){10}}
\put(31,33){\line(2,-1){30}}
\put(0,18){\makebox(0,0)[rc]{$\gr\{(12)\}$}}
\put(20,18){\makebox(0,0)[rc]{$\gr\{(13)\}$}}
\put(42,18){\makebox(0,0)[lc]{$\gr\{(23)\}$}}
\put(62,18){\makebox(0,0)[lc]{$\gr\{(123)\}$}}
\put(31,36){\makebox(0,0)[cc]{$\mathbb S_3$}}
\put(31,0){\makebox(0,0)[cc]{$\mathbb T$}}
\end{picture}
\caption{The subgroup lattice of the group $\mathbb S_3$}
\label{Sub(S_3)}
\end{center}
\end{figure}
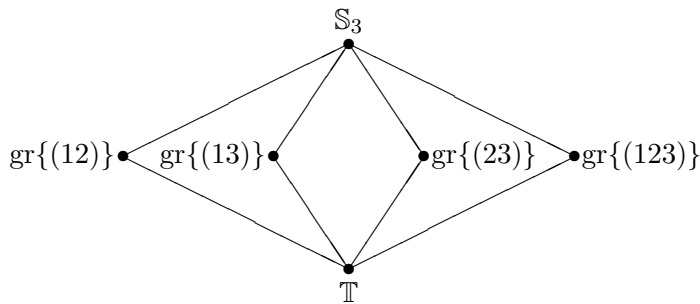

\section{The proof of the main results}
\label{proof}

\begin{proof}[Proof of Theorem~\ref{permut-3}.]
\emph{Necessity} immediately follows from Proposition~\ref{nec} and Lemma~\ref{permut-3 nec}.

\medskip

\emph{Sufficiency}. Let $\mathbf V=\mathbf{M\vee N}$ where $\mathbf M$ is one of the varieties $\mathbf T$ or $\mathbf{SL}$, while the variety $\mathbf N$ satisfies one of the identity systems~\eqref{xyz=zyx,xxy=0}--\eqref{xyz=xzy,xxy=0}. We need to verify that the variety $\mathbf V$ is a modular element in the lattice \textbf{SEM}. In view of Lemma~\ref{+-SL}, we may assume that $\mathbf M=\mathbf{SL}$, so $\mathbf V=\mathbf{SL\vee N}$. The same lemma shows that it suffices to prove that \textbf V is a modular element of the coideal $[\mathbf{SL})$ of the lattice \textbf{SEM}. In other words, we have to check that
$$
\mathbf{(V\vee Y)\wedge Z}=\mathbf{(V\wedge Z)\vee Y}
$$
for arbitrary varieties \textbf Y and \textbf Z with $\mathbf{SL\subseteq Y\subseteq Z}$. Moreover, it suffices to verify that
$$
\mathbf{(V\vee Y)\wedge Z\subseteq(V\wedge Z)\vee Y}
$$
because the opposite inclusion is evident. Thus, we need to prove that if a non-trivial identity $\mathbf{u\approx v}$ holds in $\mathbf{(V\wedge Z)\vee Y}$ then it holds in $\mathbf{(V\vee Y)\wedge Z}$ too.

So, let $\mathbf{SL\subseteq Y\subseteq Z}$ and $\mathbf{u\approx v}$ be a non-trivial identity that holds in the variety $\mathbf{(V\wedge Z)\vee Y}$. Then $\mathbf{u\approx v}$ holds in \textbf Y and there is a deduction of this identity from identities of the varieties \textbf V and \textbf Z, that is, a sequence of words
\begin{equation}
\label{deduct}
\mathbf w_0,\mathbf w_1,\dots,\mathbf w_m
\end{equation}
such that $\mathbf w_0=\mathbf u$, $\mathbf w_m=\mathbf v$ and, for each $i=0,1,\dots,m-1$, the identity $\mathbf{w_i\approx w_{i+1}}$ holds in one of the varieties \textbf V or \textbf Z. We will assume that~\eqref{deduct} is the shortest deduction of $\mathbf{u\approx v}$ from identities of \textbf V and \textbf Z. In particular, this means that the words $\mathbf w_0,\mathbf w_1,\dots,\mathbf w_m$ are pairwise distinct, there is no $i\in\{0,1,\dots,m-2\}$ such that $\mathbf w_i\approx\mathbf w_{i+1}\approx\mathbf w_{i+2}$ hold in one of the varieties \textbf V or \textbf Z, and there is no $i\in\{0,1,\dots,m-1\}$ such that $\mathbf w_i\approx\mathbf w_{i+1}$ holds in both of these two varieties. Since each of the varieties \textbf V and \textbf Z contains \textbf{SL}, Lemma~\ref{u=v in SL} implies that
\begin{equation}
\label{content}
\con(\mathbf w_0)=\con(\mathbf w_1)=\cdots=\con(\mathbf w_m).
\end{equation}
This fact and Lemma~\ref{u=v in SL} imply that if, for some $i\in\{0,1,\dots,m-1\}$, the identity $\mathbf w_i\approx\mathbf w_{i+1}$ holds in \textbf N then it holds in \textbf V too. All these observations will be used below without references.

The case $m=1$ is evident. Indeed, if $m=1$ then the identity $\mathbf{u\approx v}$ holds in one of the varieties \textbf V or \textbf Z. This identity holds also in \textbf Y, whence it holds in one of the varieties $\mathbf{V\vee Y}$ or \textbf Z, and  moreover in $\mathbf{(V\vee Y)\wedge Z}$.

Suppose that $m=2$. By symmetry, we may assume that $\mathbf u\approx \mathbf w_1$ holds in \textbf V and $\mathbf w_1\approx\mathbf v$ holds in \textbf Z. Then $\mathbf w_1\approx\mathbf v\approx\mathbf u$ hold in \textbf Y. We see that the identities $\mathbf u\approx\mathbf w_1$ and $\mathbf w_1\approx\mathbf v$ hold in $\mathbf{V\vee Y}$ and \textbf Z respectively, whence $\mathbf{u\approx v}$ holds in $\mathbf{(V\vee Y)\wedge Z}$.

\smallskip

\textsl{At the rest part of the proof we suppose that} $m\ge3$.

\smallskip

First of all, we consider one very special but important partial case.

\begin{lemma}
\label{ZVZ}
If $m=3$, the identities $\mathbf w_0\approx\mathbf w_1$ and $\mathbf w_2\approx\mathbf w_3$ hold in the variety $\mathbf Z$, and the identity $\mathbf w_1\approx\mathbf w_2$ holds in the variety $\mathbf V$ then the identity $\mathbf{u\approx v}$ holds in the variety $\mathbf{(V\vee Y)\wedge Z}$.
\end{lemma}

\begin{proof}
Recall that the identity $\mathbf{u\approx v}$ holds in \textbf Y, $\mathbf u=\mathbf w_0$ and $\mathbf v=\mathbf w_3$. Since $\mathbf{Y\subseteq Z}$, we have that $\mathbf w_1\approx\mathbf w_0\approx\mathbf w_3\approx\mathbf w_2$ hold in \textbf Y. Therefore, the identity $\mathbf w_1\approx\mathbf w_2$ holds in the variety $\mathbf{V\vee Y}$. It remains to take into account that the identities $\mathbf w_0\approx\mathbf w_1$ and $\mathbf w_2\approx\mathbf w_3$ hold in \textbf Z.
\end{proof}

Put
\begin{align*}
Z&=\{\mathbf w\in F\mid\ \text{the identity}\ \mathbf w\approx 0\ \text{holds in the variety}\ \mathbf N\},\\
L&=\{\mathbf w\in F\mid\ \text{the word}\ \mathbf w\ \text{is linear and}\ \mathbf w\notin Z\},\\
S&=\{\mathbf w\in F\mid \mathbf w\notin Z\cup L\}.
\end{align*}
Lemma~\ref{permut-3 nec} readily implies that if $S\ne\varnothing$ then any word from $S$ is similar to:
\begin{itemize}
\item[$\bullet$] one of the words $x^2$ or $xyx$ whenever \textbf V satisfies~\eqref{xyz=zyx,xxy=0};
\item[$\bullet$] the word $x^2$ whenever \textbf V satisfies~\eqref{xyz=yzx,xxy=0};
\item[$\bullet$] one of the words $x^2$ or $x^2y$ whenever \textbf V satisfies~\eqref{xyz=yxz,xyy=0};
\item[$\bullet$] one of the words $x^2$ or $xy^2$ whenever \textbf V satisfies~\eqref{xyz=xzy,xxy=0}.
\end{itemize}
Let $S_k$ be the set of all words from $S$ depending on $k$ letters. We have the following

\begin{remark}
\label{words from S}
$S=S_1\cup S_2$; if $\mathbf u\in S_1$ then $\mathbf u$ is similar to $x^2$; if $\mathbf u,\mathbf v\in S_2$ then $\mathbf u$ and $\mathbf v$ are similar.\qed
\end{remark}

It is evident that each of the words $\mathbf w_0,\mathbf w_1,\dots,\mathbf w_m$ lies in exactly one of the sets $Z$, $L$ or $S$. It is clear also that if $\mathbf w',\mathbf w''\in Z$ then $\mathbf w'\approx\mathbf w''$ in \textbf N; if, besides that, $\con(\mathbf w')=\con(\mathbf w'')$ then $\mathbf w'\approx\mathbf w''$ in \textbf V. In particular, if $\mathbf w_i,\mathbf w_j\in Z$ for some $i,j\in\{0,1,\dots,m\}$ then $\mathbf w_i\approx\mathbf w_j$ in \textbf V. All these observations will be used below without references.

One can provide some properties of the sequence~\eqref{deduct}.

\begin{lemma}
\label{Z}
If $\mathbf w_i,\mathbf w_j\in Z$ for some $0\le i<j\le m$ then $j=i+1$. In particular, the sequence~\textup{\eqref{deduct}} contains at most two words from the set $Z$.
\end{lemma}

\begin{proof}
The identity $\mathbf w_i\approx\mathbf w_j$ holds in the variety \textbf V. If $j>i+1$ then the sequence of words
$$
\mathbf w_0,\mathbf w_1,\dots,\mathbf w_i,\mathbf w_j,\dots,\mathbf w_m
$$
is a deduction of the identity $\mathbf{u\approx v}$ from identities of the varieties \textbf V and \textbf Z shorter than the sequence~\eqref{deduct}.
\end{proof}

\begin{lemma}
\label{S}
If $\mathbf w_i\in S_1$ for some $i\in\{0,1,\dots,m\}$ then:
\begin{itemize}
\item[\textup{(i)}] either $i=0$ or $i=m$;
\item[\textup{(ii)}] if $i=0$ \textup[respectively $i=m$\textup] then the identity $\mathbf w_0\approx\mathbf w_1$ \textup[respectively $\mathbf w_{m-1}\approx\mathbf w_m$\textup] holds in the variety $\mathbf Z$.
\end{itemize}
\end{lemma}

\begin{proof}
In view of Remark~\ref{words from S}, we may assume without loss of generality that $\mathbf w_i= x^2$. Suppose that $0<i<m$. Then $\mathbf w_i\approx\mathbf w_j$ holds in \textbf V for some $j\in\{i-1,i+1\}$. In particular, $\mathbf w_i\approx\mathbf w_j$ in \textbf N. Since $\con(\mathbf w_j)=\con(\mathbf w_i)=\{x\}$ and $\mathbf w_j\ne\mathbf w_i$, we have that $\mathbf w_j=x^k$ for some $k\ne2$. It is evident that $x^k\approx x^2$ implies $x^2\approx 0$ in any nil-variety. Therefore, $\mathbf w_i\approx 0$ holds in \textbf N contradicting the claim $\mathbf w_i\in S$. The claim~(i) is proved. Analogous arguments show that if $i=0$ [respectively $i=m$] then the identity $\mathbf w_0\approx\mathbf w_1$ [respectively $\mathbf w_{m-1}\approx\mathbf w_m$] fails in \textbf V and therefore, holds in \textbf Z. The claim~(ii) is proved as well.
\end{proof}

\begin{lemma}
\label{L}
If the sequence~\textup{\eqref{deduct}} does not contain a word from the set $S_2$ and $\mathbf w_i\in L$ for some $i\in\{1,\dots,m-1\}$ then either $\mathbf w_{i-1}\in L$ or $\mathbf w_{i+1}\in L$.
\end{lemma}

\begin{proof}
Arguing by contradiction, suppose that $\mathbf w_{i-1},\mathbf w_{i+1}\in Z\cup S$. If $\mathbf w_{i-1}\in S$ then $\mathbf w_{i-1}\in S_1$ and Lemma~\ref{S} applies with the conclusion that the identity $\mathbf w_{i-1}\approx\mathbf w_i$ holds in \textbf Z. Let now $\mathbf w_{i-1}\in Z$. Then the variety \textbf N satisfies the identity $\mathbf w_{i-1}\approx 0$. On the other hand, the identity $\mathbf w_i\approx 0$ fails in \textbf N because $\mathbf w_i\in L$. Therefore, the identity $\mathbf w_{i-1}\approx\mathbf w_i$ fails in \textbf V, whence it holds in \textbf Z. We see that $\mathbf w_{i-1}\approx\mathbf w_i$ holds in \textbf Z in any case. Analogous arguments show that $\mathbf w_i\approx\mathbf w_{i+1}$ holds in \textbf Z. Thus, $\mathbf w_{i-1}\approx\mathbf w_i\approx\mathbf w_{i+1}$ hold in \textbf Z that is impossible.
\end{proof}

\begin{lemma}
\label{LL-Zfree}
If the sequence~\textup{\eqref{deduct}} does not contain a word from the set $S_2$ and $\mathbf w_i,\mathbf w_{i+1}\in L$ for some $i\in\{0,1,\dots,m-1\}$ then the sequence~\textup{\eqref{deduct}} does not contain words from the set $Z$.
\end{lemma}

\begin{proof}
Suppose that $i+1<m$ and $\mathbf w_{i+2}\in Z$. Then the variety \textbf N satisfies the identity $\mathbf w_{i+2}\approx 0$ but does not satisfy the identity $\mathbf w_{i+1}\approx 0$. Hence the identity $\mathbf w_{i+1}\approx\mathbf w_{i+2}$ is false in \textbf N, and moreover in \textbf V. Therefore, $\mathbf w_{i+1}\approx\mathbf w_{i+2}$ holds in \textbf Z, whence $\mathbf w_i\approx\mathbf w_{i+1}$ holds in \textbf V. Clearly, $\mathbf w_i=\sigma[\mathbf w_{i+1}]$ for some permutation $\sigma$ on the set $\con(\mathbf w_{i+1})$. Then $\mathbf w_i=\sigma(\mathbf w_{i+1})\approx\sigma(\mathbf w_{i+2})$ holds in \textbf Z. Since $\mathbf w_{i+2}\approx 0$ holds in \textbf N, we have that $\sigma[\mathbf w_{i+2}]\approx 0$ holds in \textbf N too. Thus $\sigma[\mathbf w_{i+2}]\approx\mathbf w_{i+2}$ holds in \textbf N and therefore, in \textbf V. Recall that $m\ge3$. Hence either $i>0$ or $i+2<m$. If $i>0$ then $\mathbf w_{i-1}\approx\mathbf w_i\approx \sigma[\mathbf w_{i+2}]$ hold in \textbf Z and the sequence of words
$$
\mathbf w_0,\mathbf w_1,\dots,\mathbf w_{i-1},\sigma[\mathbf w_{i+2}],\mathbf w_{i+2},\dots,\mathbf w_m
$$
is a deduction of the identity $\mathbf{u\approx v}$ from identities of the varieties \textbf V and \textbf Z shorter than the sequence~\eqref{deduct}. Further, if $i+2<m$ then $\sigma[\mathbf w_{i+2}]\approx\mathbf w_{i+2}\approx\mathbf w_{i+3}$ hold in \textbf V and we also have a deduction of the identity $\mathbf{u\approx v}$ from identities of the varieties \textbf V and \textbf Z shorter than the sequence~\eqref{deduct}, namely the sequence of words
$$
\mathbf w_0,\mathbf w_1,\dots,\mathbf w_i,\sigma[\mathbf w_{i+2}],\mathbf w_{i+3},\dots,\mathbf w_m.
$$
Both these cases are impossible. We prove that if $i+1<m$ then $\mathbf w_{i+2}\notin Z$. By symmetry, if $i>0$ then $\mathbf w_{i-1}\notin Z$.

Let now $\mathbf w_j,\mathbf w_{j+1},\dots,\mathbf w_{j+k}$ (where $0\le j<j+k\le m$) be the maximal subsequence of the sequence~\eqref{deduct} consisting of words from the set $L$. In other words, $\mathbf w_j,\mathbf w_{j+1},\dots,\mathbf w_{j+k}\in L$, $\mathbf w_{j+k+1}\notin L$ whenever $j+k<m$, and $\mathbf w_{j-1}\notin L$ whenever $j>0$. Suppose that $j+k<m$. As we have seen in the previous paragraph, this implies that $\mathbf w_{j+k+1}\notin Z$, whence $\mathbf w_{j+k+1}\in S$. Now Lemma~\ref{S}(i) applies, and we conclude that $j+k+1=m$. Analogously, if $j>0$ then $j=1$ and $\mathbf w_0\in S$. We see that words from the set $Z$ are absent in the sequence~\eqref{deduct}.
\end{proof}

Further considerations are divided into two cases.

\smallskip

\emph{Case} 1: the sequence~\eqref{deduct} does not contain a word from the set $S_2$. Suppose at first that the sequence~\eqref{deduct} does not contain adjacent words from the set $L$. Lemmas~\ref{S}(i) and~\ref{L} imply that $\mathbf w_1,\dots,\mathbf w_{m-1}\in Z$ in this case, whence $\mathbf w_1\approx \cdots\approx\mathbf w_{m-1}$ in $\mathbf V$. Therefore, $m-1=2$, that is, $m=3$. Since the identity $\mathbf w_1\approx\mathbf w_2$ holds in $\mathbf V$, the identities $\mathbf w_0\approx\mathbf w_1$ and $\mathbf w_2\approx\mathbf w_3$ hold in $\mathbf Z$. Now Lemma~\ref{ZVZ} applies and we are done.

Suppose now that the sequence~\eqref{deduct} contains adjacent words from the set $L$ but does not contain three words in row from this set. By Lemma~\ref{LL-Zfree} the sequence~\eqref{deduct} does not contain words from the set $Z$. Then Lemma~\ref{S}(i) implies that $m=3$, $\mathbf w_0,\mathbf w_3\in S$ and $\mathbf w_1,\mathbf w_2\in L$. By Lemma~\ref{S}(ii) the identities $\mathbf w_0\approx\mathbf w_1$ and $\mathbf w_2\approx\mathbf w_3$ hold in \textbf Z. Therefore, $\mathbf w_1\approx\mathbf w_2$ holds in \textbf V and Lemma~\ref{ZVZ} applies again.

Finally, suppose that the sequence~\eqref{deduct} contains three words in row from the set $L$. Let $\mathbf w_i,\mathbf w_{i+1},\mathbf w_{i+2}\in L$ for some $i\in\{0,1,\dots,m-2\}$. Put $k=\ell(\mathbf w_i)$. In view of~\eqref{content}, we have $\ell(\mathbf w_{i+1})=\ell(\mathbf w_{i+2})=k$. One of the identities $\mathbf w_i\approx\mathbf w_{i+1}$ or $\mathbf w_{i+1}\approx\mathbf w_{i+2}$ holds in the variety \textbf V, whence it holds in \textbf N. 

Suppose that $k=2$. The variety \textbf N is commutative in this case. Then \textbf N is a modular element of \textbf{SEM} by Proposition~\ref{commut} and therefore, \textbf V has the same property by Lemma~\ref{+-SL}.

Let now $k\ge 3$. We may assume without loss of generality that $\con(\mathbf w_j)=\{x_1,x_2,\dots,x_k\}$ for all $j=0,1,\dots,m$. According to Lemma~\ref{LL-Zfree} the sequence~\eqref{deduct} does not contain words from the set $Z$. Further, by Remark~\ref{words from S} if $\mathbf w\in S$ then $\mathbf w$ is similar to $x^2$, whence $\con(\mathbf w)\ne\{x_1,x_2,\dots,x_k\}$. Hence~\eqref{deduct} does not contain words from the set $S$. Therefore, $\mathbf w_0,\mathbf w_1,\dots,\mathbf w_m\in L$.{\sloppy

}

Suppose that $k=3$. Since $\mathbf w_i\approx\mathbf w_{i+1}$ holds in \textbf Z for some $i\in\{0,1,\dots,m-1\}$, the variety \textbf Z satisfies a permutational identity of length~3. In other words, the group $\Perm_3(\mathbf Z)$ contains some non-trivial permutation $\sigma$. It is clear that $\mathbf w_m=\tau[\mathbf w_0]$ for some permutation $\tau\in\mathbb S_3$. The permutation $\tau$ is non-trivial because the identity $\mathbf{u\approx v}$ is non-trivial. If $\tau\in\Perm_3(\mathbf Z)$ then the identity $\mathbf w_0\approx\mathbf w_m$ (that is, $\mathbf{u\approx v}$) holds in the variety \textbf Z and therefore, in $\mathbf{(V\vee Y)\wedge Z}$. Suppose now that $\tau\notin\Perm_3(\mathbf Z)$. The identity $\mathbf{u\approx v}$ has the form $\mathbf u\approx \tau[\mathbf u]$. Recall that this identity holds in the variety \textbf Y, whence $\tau\in\Perm_3(\mathbf Y)$. Besides that, $\sigma\in\Perm_3(\mathbf Y)$ because $\mathbf{Y\subseteq Z}$. Sinse $\tau\notin\Perm_3(\mathbf Z)$, we have that $\sigma$ and $\tau$ generate distinct non-trivial subgroups of $\mathbb S_3$. Therefore, the subgroup of $\mathbb S_3$ generated by these two permutations coincides with $\mathbb S_3$ (see Fig.~\ref{Sub(S_3)}), that is, $\Perm_3(\mathbf Y)=\mathbb S_3$. In other words, \textbf Y satisfies all permutational identities of length~3. This means that, for each $i=0,1,\dots,m-1$, the identity $\mathbf w_i\approx\mathbf w_{i+1}$ holds in one of the varieties $\mathbf{V\vee Y}$ or \textbf Z, whence $\mathbf{u\approx v}$ holds in $\mathbf{(V\vee Y)\wedge Z}$.

Finally, let $k\ge 4$. The identity $\mathbf{u\approx v}$, that is, $\mathbf w_0\approx\mathbf w_m$ is a permutational identity of length $k$. In other words, there is a permutation $\xi\in\mathbb S_k$ with $\mathbf v=\xi[\mathbf u]$. If \textbf N satisfies one of the identity systems~\eqref{xyz=zyx,xxy=0} or~\eqref{xyz=yzx,xxy=0} then the claim~1(iii) of Lemma~\ref{permut conseq} applies with the conclusion that $\xi\in\Perm_k(\mathbf N)$. This means that the identity $\mathbf{u\approx v}$ holds in \textbf N and therefore, in \textbf V. In this case $\mathbf{u\approx v}$ holds in $\mathbf{V\vee Y}$, and moreover in $(\mathbf{V\vee Y})\wedge\mathbf Z$. By symmetry, it remains to consider the case when \textbf N satisfies the identity system~\eqref{xyz=yxz,xyy=0}.

Suppose that $k=4$. In view the claim~1(i) of Lemma~\ref{permut conseq}, $\Perm_4(\mathbf N)\supseteq\Stab_4(4)$. Besides that, \textbf N satisfies the identity $xyzt\approx xzty$, whence the group $\Perm_4(\mathbf N)$ contains the permutation $(234)$. Since this permutation does not lie in $\Stab_4(4)$ and $\Stab_4(4)$ is a maximal proper subgroup in $\mathbb S_4$, we have that $\Perm_4(\mathbf N)=\mathbb S_4$. Further, let $k\ge 5$. Then the claims~1(i) and~2 of Lemma~\ref{permut conseq} imply that the group $\Perm_k(\mathbf N)$ contains both the groups $\Stab_k(1)$ and $\Stab_k(k)$, whence $\Perm_k(\mathbf N)=\mathbb S_k$. Thus, the last equality holds for any $k\ge 4$. In particular, $\xi\in\Perm_k(\mathbf N)$ in this case. As we have already seen in the previous paragraph, this implies the requirement conclusion. 

\smallskip

\emph{Case} 2: the sequence~\eqref{deduct} contains a word from the set $S_2$. All words from $S_2$ depend on two letters. In view of~\eqref{content}, we may assume that $\con(\mathbf w_i)=\{x,y\}$ for each $i=0,1,\dots,m$. In particular, the sequence~\eqref{deduct} does not contain words from the set $S_1$. Besides that, Remark~\ref{words from S} shows that if $\mathbf w_i,\mathbf w_j\in S_2$ for some $0\le i<j\le m$ then the words $\mathbf w_i$ and $\mathbf w_j$ are similar. Since words from $S_2$ depend on two letters, this implies that the sequence~\eqref{deduct} contains at most two words from the set $S_2$. All these observations will be used below without references. We need two additional lemmas.

\begin{lemma}
\label{w_i in L}
If the sequence~\eqref{deduct} contains a word from the set $S_2$, $\mathbf w_i\in L$ for some $i\in\{0,1,\dots,m\}$ and the variety $\mathbf V$ satisfies a non-trivial identity of the form $\mathbf w_i\approx\mathbf w$ then the variety $\mathbf V$ is a modular element in $\mathbf{SEM}$.
\end{lemma}

\begin{proof}
The word $\mathbf w_i$ is linear and $\con(\mathbf w_i)=\{x,y\}$. Therefore, $\mathbf w_i\in\{xy,yx\}$. We may assume without loss of generality that $\mathbf w_i= xy$. The identity $xy\approx\mathbf w$ holds in the variety \textbf N. If $\mathbf w= yx$ then the variety \textbf N is a modular element in \textbf{SEM} by Proposition~\ref{commut}, whence the variety \textbf V has the same property by Lemma~\ref{+-SL}. Suppose now that $\mathbf w\ne yx$. Then $\ell(\mathbf w)\ne2$. By Lemma~\ref{x_1x_2...x_n=w} $xy\approx 0$ holds in \textbf N. Then the variety $\mathbf V=\mathbf{SL\vee N}$ is a neutral element of the lattice \textbf{SEM} by~\cite[Proposition~4.1]{Volkov-05}. Therefore, the variety \textbf V is a modular element in \textbf{SEM}.
\end{proof}

\begin{lemma}
\label{w_i=word-from-S_2 in V}
If the variety $\mathbf V$ satisfies the identity $\mathbf w_i\approx\mathbf w_{i+1}$ for some $i\in\{0,1,\dots,m-1\}$ and one of the words $\mathbf w_i$ or $\mathbf w_{i+1}$ lies in $S_2$ then either the variety $\mathbf V$ is a modular element in $\mathbf{SEM}$ or both the words $\mathbf w_i$ and $\mathbf w_{i+1}$ lie in $S_2$.
\end{lemma}

\begin{proof}
We may assume without loss of generality that $\mathbf w_i\in S_2$. If $\mathbf w_{i+1}\in Z$ then $\mathbf w_i\approx\mathbf w_{i+1}\approx 0$ in \textbf N contradicting the claim $\mathbf w_i\in S$. Further, if $\mathbf w_{i+1}\in L$ then Lemma~\ref{w_i in L} applies with the conclusion that the variety \textbf V is a modular element in \textbf{SEM}. Finally, if $\mathbf w_{i+1}\notin Z\cup L$ then $\mathbf w_{i+1}\in S$. Since the sequence~\eqref{deduct} does not contain words from the set $S_1$, we have that $\mathbf w_{i+1}\in S_2$.
\end{proof}

Suppose now that the sequence~\eqref{deduct} contains only one word from the set $S_2$. Namely, let $\mathbf w_i\in S_2$. Suppose that $i\in\{1,\dots,m-1\}$. Then the variety \textbf V satisfies one of the identities $\mathbf w_{i-1}\approx\mathbf w_i$ or $\mathbf w_i\approx\mathbf w_{i+1}$. Since $\mathbf w_{i-1},\mathbf w_{i+1}\notin S_2$, Lemma~\ref{w_i=word-from-S_2 in V} implies that the variety \textbf V is a modular element in \textbf{SEM}. It remains to consider the case when either $i=0$ or $i=m$. By symmetry, we may suppose that $i=0$. Then $\mathbf w_j\notin S$ for all $j=1,2,\dots,m$. If the identity $\mathbf w_0\approx\mathbf w_1$ holds in \textbf V then we may apply Lemma~\ref{w_i=word-from-S_2 in V} and conclude that the variety \textbf V is a modular element in \textbf{SEM}. Suppose now that the identity $\mathbf w_0\approx\mathbf w_1$ holds in \textbf Z. Then $\mathbf w_1\approx\mathbf w_2$ in \textbf V. In view of Lemma~\ref{w_i in L}, we may assume that $\mathbf w_2\notin L$. Therefore, $\mathbf w_2\in Z$. Further, $\mathbf w_2\approx\mathbf w_3$ in $\mathbf Z$. If $m=3$ then Lemma~\ref{ZVZ} applies with the required conclusion. Otherwise, $\mathbf w_3\approx\mathbf w_4$ in \textbf V. Lemma~\ref{w_i in L} allows us to assume that $\mathbf w_4\notin L$, whence $\mathbf w_4\in Z$. Therefore, $\mathbf w_2\approx\mathbf w_4$ in \textbf V. This means that the sequence of words
$$
\mathbf w_0,\mathbf w_1,\mathbf w_2,\mathbf w_4,\dots,\mathbf w_m
$$
is a deduction of the identity $\mathbf{u\approx v}$ from identities of the varieties \textbf V and \textbf Z shorter than~\eqref{deduct}.

Finally, suppose that the sequence~\eqref{deduct} contains two words from the set $S_2$. Suppose at first that $\mathbf w_i\in S_2$ for some $i\in\{1,\dots,m-1\}$. Then the variety \textbf V satisfies the identity $\mathbf w_i\approx\mathbf w_j$ for some $j\in\{i-1,i+1\}$. In view of Lemma~\ref{w_i=word-from-S_2 in V}, we may assume that $\mathbf w_j\in S_2$ in this case. Suppose now that $\mathbf w_i\notin S_2$ whenever $i\in\{1,\dots,m-1\}$. Then $\mathbf w_0,\mathbf w_m\in S_2$. If, besides that, at least one of the identities $\mathbf w_0\approx\mathbf w_1$ or $\mathbf w_{m-1}\approx\mathbf w_m$ holds in the variety \textbf V then Lemma~\ref{w_i=word-from-S_2 in V} applies and we conclude that \textbf V is a modular element in \textbf{SEM}. We see that either $\mathbf w_0,\mathbf w_m\in S_2$ and the identities $\mathbf w_0\approx\mathbf w_1$ and $\mathbf w_{m-1}\approx\mathbf w_m$ hold in \textbf Z or $\mathbf w_i,\mathbf w_{i+1}\in S_2$ for some $i\in\{0,1,\dots,m-1\}$ and $\mathbf w_i\approx\mathbf w_{i+1}$ holds in \textbf V.

Suppose that $\mathbf w_0,\mathbf w_m\in S_2$ and the identities $\mathbf w_0\approx\mathbf w_1$ and $\mathbf w_{m-1}\approx\mathbf w_m$ hold in \textbf Z. Then $\mathbf w_1,\mathbf w_{m-1}\notin S$. The identities $\mathbf w_1\approx\mathbf w_2$ and $\mathbf w_{m-2}\approx\mathbf w_{m-1}$ hold in the variety \textbf V. Lemma~\ref{w_i in L} permits to assume that $\mathbf w_1,\mathbf w_{m-1}\notin L$. Since $\mathbf w_1,\mathbf w_{m-1}\notin S$, we have $\mathbf w_1,\mathbf w_{m-1}\in Z$. Then $\mathbf w_1\approx\mathbf w_{m-1}$ holds in $\mathbf V$. This means that the sequence of words
$$
\mathbf w_0,\mathbf w_1,\mathbf w_{m-1},\mathbf w_m
$$
is a deduction of the identity $\mathbf{u\approx v}$ from identities of the varieties \textbf V and \textbf Z. If $m>3$ then this deduction shorter than~\eqref{deduct}, while if $m=3$ then Lemma~\ref{ZVZ} applies.

It remains to consider the case when $\mathbf w_i,\mathbf w_{i+1}\in S_2$ for some $i\in\{0,1,\dots,m-1\}$ and $\mathbf w_i\approx\mathbf w_{i+1}$ holds in \textbf V. Suppose at first that $m-3<i<2$. Since $m\ge3$, we have $m=3$ and $i=1$. Then the identity $\mathbf w_1\approx\mathbf w_2$ holds in \textbf V, whence the identities $\mathbf w_0\approx\mathbf w_1$ and $\mathbf w_2\approx\mathbf w_3$ hold in \textbf Z. Now Lemma~\ref{ZVZ} applies. So, we may assume that either $i\ge 2$ or $i\le m-3$. By symmetry, it suffices to consider the former case. So, let $i\ge2$. Then the identity $\mathbf w_{i-1}\approx\mathbf w_i$ holds in \textbf Z, while the identity $\mathbf w_{i-2}\approx\mathbf w_{i-1}$ holds in \textbf V. Clearly, $\mathbf w_{i-1}\notin S$. Besides that, Lemma~\ref{w_i in L} allows us to assume that $\mathbf w_{i-1}\notin L$. Therefore, $\mathbf w_{i-1}\in Z$. Since $\mathbf w_{i-2}\approx\mathbf w_{i-1}$ holds in \textbf V, this identity holds also in \textbf N. Hence $\mathbf w_{i-2}\approx\mathbf w_{i-1}\approx 0$ hold in \textbf N, that is $\mathbf w_{i-2}\in Z$. Let $\pi$ be a unique non-trivial permutation on the set $\{x,y\}$. Remark~\ref{words from S} implies that $\mathbf w_{i+1}=\pi[\mathbf w_i]$. Since the variety \textbf Z satisfies the identity $\mathbf w_{i-1}\approx\mathbf w_i$, it satisfies also the identity $\pi[\mathbf w_{i-1}]\approx\mathbf\pi[\mathbf w_i]$, that is $\pi[\mathbf w_{i-1}]\approx\mathbf w_{i+1}$. Further, the variety \textbf N satisfies the identity $\pi[\mathbf w_{i-1}]\approx 0$ because it satisfies $\mathbf w_{i-1}\approx 0$. Thus, the identity $\mathbf w_{i-2}\approx\pi[\mathbf w_{i-1}]$ holds in \textbf N and therefore, in \textbf V. We see that the sequence of words
$$
\mathbf w_0,\mathbf w_1,\dots,\mathbf w_{i-2},\pi[\mathbf w_{i-1}],\mathbf w_{i+1},\dots,\mathbf w_m
$$
is a deduction of the identity $\mathbf{u\approx v}$ from identities of the varieties \textbf V and \textbf Z shorter than~\eqref{deduct}.

We complete the proof of Theorem~\ref{permut-3}.
\end{proof}

\begin{proof}[Proof of Proposition~\ref{modular not cancellable}.]
Put
$$
\mathbf V=\var\{xyzt\approx xyx\approx x^2\approx 0,\,x_1x_2x_3\approx x_{1\rho}x_{2\rho}x_{3\rho}\}.
$$
The variety \textbf V satisfies one of the identity systems~\eqref{xyz=zyx,xxy=0}--\eqref{xyz=xzy,xxy=0}. According to Theorem~\ref{permut-3} \textbf V is a modular element in the lattice \textbf{SEM}. It remains to check that \textbf V is not a cancellable element of this lattice. Let $\sigma$ and $\tau$ be non-trivial permutations from $\mathbb S_3$ such that the groups $\gr\{\rho\}$, $\gr\{\sigma\}$ and $\gr\{\tau\}$ are pairwise distinct. Put
\begin{align*}
\mathbf X={}&\var\{xyzt\approx xyx\approx x^2\approx 0,\,x_1x_2x_3\approx x_{1\sigma}x_{2\sigma}x_{3\sigma}\},\\
\mathbf Y={}&\var\{xyzt\approx xyx\approx x^2\approx 0,\,x_1x_2x_3\approx x_{1\tau}x_{2\tau}x_{3\tau}\}.
\end{align*}
Clearly, $\mathbf X\ne\mathbf Y$. It is evident that $\Perm_3(\mathbf V)=\gr\{\rho\}$, $\Perm_3(\mathbf X)=\gr\{\sigma\}$ and $\Perm_3(\mathbf Y)=\gr\{\tau\}$. Lemma~\ref{Perm_n}(ii) and Fig.~\ref{Sub(S_3)} imply that
$$
\Perm_3(\mathbf{V\wedge X})=\Perm_3(\mathbf V)\vee\Perm_3(\mathbf X)=\mathbb S_3.
$$
This means that the variety $\mathbf{V\wedge X}$ satisfies all permutational identities of length~3. Hence $\mathbf{V\wedge X\subseteq Y}$ and therefore, $\mathbf{V\wedge X\subseteq V\wedge Y}$. The opposite inclusion is verified analogously, whence $\mathbf{V\wedge X}=\mathbf{V\wedge Y}$.

It remains to verify that $\mathbf{V\vee X}=\mathbf{V\vee Y}$. Suppose that an identity $\mathbf{u\approx v}$ holds in the variety $\mathbf{V\vee X}$. We are going to check that then this identity holds in \textbf Y. We can assume that the identity $\mathbf{u\approx v}$ is non-trivial because the required conclusion is evident in the contrary case. If each of the words \textbf u and \textbf v either is non-linear or has the length $>3$ then $\mathbf u\approx 0\approx\mathbf v$ hold in \textbf Y. Let now \textbf u and \textbf v be linear words of length $\le 3$. If $\mathbf{\con(u)\ne\con(v)}$ then $\mathbf u\approx 0\approx\mathbf v$ hold in \textbf Y again because \textbf Y is a nil-variety. Therefore, we can assume that $\con(\mathbf u)=\con(\mathbf v)$. In particular, this means that $\ell(\mathbf u)=\ell(\mathbf v)$. Clearly, $\ell(\mathbf u)>1$ because the identity $\mathbf{u\approx v}$ is trivial otherwise. If $\ell(\mathbf u)=2$ then $\mathbf{u\approx v}$ is the commutative law. But then $\mathbf{u\approx v}$ is false in \textbf V, and moreover in $\mathbf{V\vee X}$. Finally, suppose that $\ell(\mathbf u)=3$. Then $\mathbf{u\approx v}$ is a permutational identity of length~3. Lemma~\ref{Perm_n}(i) and Fig.~\ref{Sub(S_3)} imply that 
$$
\Perm_3(\mathbf{V\vee X})=\Perm_3(\mathbf V)\wedge\Perm_3(\mathbf X)=\mathbb T.
$$
Therefore, the identity $\mathbf{u\approx v}$ is trivial. We have a contradiction. Thus, $\mathbf{u\approx v}$ holds in \textbf Y. Hence $\mathbf{V\vee X\supseteq Y}$ and therefore, $\mathbf{V\vee X\supseteq V\vee Y}$. The opposite inclusion is verified analogously, whence $\mathbf{V\vee X}=\mathbf{V\vee Y}$ and we are done.
\end{proof}

\subsection*{Acknowledgments.} The authors express their thanks to Dr. V.\,Shaprynski\v{\i} for fruitful discussions.

\end{document}